\documentclass[12pt,a4paper,draft]{article}
\usepackage{amsmath,amsfonts,amssymb,amscd,amsthm}
\usepackage{latexsym}
\usepackage[T2A]{fontenc}
\usepackage[cp1251]{inputenc}
\usepackage{graphicx}

\theoremstyle{plain} 
\newtheorem{theorem}{Theorem}[section]
\newtheorem{lemma}[theorem]{Lemma}

\newtheorem{proposition}[theorem]{Proposition}

\theoremstyle{definition} 
\newtheorem{definition}[theorem]{Definition}
\theoremstyle{remark} 
\newtheorem{remark}[theorem]{Remark}

\newcommand{\bbC}{\mathbb{C}}
\newcommand{\bbD}{\mathbb{D}}
\newcommand{\bbN}{\mathbb{N}}

\newcommand{\bfm}{\mathbf{m}}

\newcommand{\mcP}{\mathcal{P}}
\newcommand{\mcU}{\mathcal{U}}
\newcommand{\mcB}{\mathcal{B}}

\begin{document}
\title{Parameterizing degree~$n$ polynomials by multipliers of periodic orbits}
\author{Igors Gorbovickis}
\maketitle

\begin{abstract}
We present the following result: consider the space of complex polynomials of degree $n\ge 3$ with $n-1$ distinct marked periodic orbits of given periods. Then this space is irreducible and the multipliers of the marked periodic orbits considered as algebraic functions on the above mentioned space, are algebraically independent over~$\bbC$. Equivalently, this means that at its generic point, the moduli space of degree $n$ polynomial maps can be locally parameterized by the multipliers of $n-1$ arbitrary distinct periodic orbits. A detailed proof of this result (together with a proof of a more general statement) is given in~\cite{mult_n}. In this exposition we substitute some of the technical lemmas from~\cite{mult_n} with more geometric arguments.
\\[1ex]
\noindent MSC 37F10, 37F05
\end{abstract}

\section{Introduction}
A key point in studying moduli spaces of degree $n$ rational or polynomial maps of the Riemann sphere $\hat{\bbC}$ is the choice of a parameterization. The idea of using the multipliers of the fixed points of a map as the parameters of the moduli space appears naturally in many works on the subject. Notably, in~\cite{Milnor_M2} J.~Milnor used the multipliers of the fixed points to parameterize the moduli space of degree $2$ rational maps. Using this parameterization he proved that this moduli space is isomorphic to $\bbC^2$. In the polynomial case local parameterization by the multipliers of the fixed points was studied, for instance, in~\cite{Sugiyama},~\cite{Zarhin_2}.

In the following discussion let $M$ be the moduli space of either rational or polynomial maps of degree~$n$. 
As a generalization of the approach described above, instead of the multipliers at the fixed points one can try to use the multipliers of periodic orbits as the local parameters on the moduli space $M$. 
It is not hard to see that the map from $M$ to the multipliers of the chosen periodic orbits is defined in a neighborhood of a generic point of $M$. The main difficulty is to show that this map is a local diffeomorphism, when the number of the chosen periodic orbits is equal to the dimension of $M$. Since multipliers are (multiple valued) algebraic maps on $M$, this leads to the question whether there exist ``hidden'' algebraic relations between the multipliers of the chosen periodic orbits. 
In other words, are the chosen multipliers algebraically independent over~$\bbC$, if we view those multipliers as (multiple valued) functions on $M$?

In~\cite{McMullen} McMullen proved that if $n\ge 2$ and $M$ is the moduli space of degree $n$ rational maps then, except for the flexible Latt\`{e}s maps, an element of $M$ is determined up to finitely many choices by the multipliers of \textit{all} of its periodic orbits. This implies that one can always choose $2n-2=\dim(M)$ distinct periodic orbits whose multipliers, considered as (multiple valued) functions on $M$, are algebraically independent over~$\bbC$.

In this paper we focus on the case when $M$ is the moduli space of degree $n$ polynomial maps and we show that the multipliers of any $n-1$ distinct periodic orbits are algebraically independent over~$\bbC$. The best previous result in this direction belongs to Yu.~Zarhin~\cite{Zarhin_1}, who proved the independence of multipliers with a strong restriction on both the \emph{number} and the \emph{periods} of the corresponding periodic orbits.

Finally, let us mention that a wholly different approach to the subject is given by A.~Epstein's Transversality Principles in Holomorphic Dynamics~\cite{Epstein_1}.
The techniques used in this paper are not directly related, and it is an intriguing question whether our results could be obtained using Epstein's methods. 

\textbf{Acknowledgments:} I wish to express my thanks to Adam Epstein and Michael Yampolsky for insightful and stimulating discussions.

\subsection{Statement of the main result}

It is easy to see that the change of coordinates $z\mapsto z+c$ preserves the class of monic polynomials of degree $n$ and does not change multipliers of periodic orbits. We say that two polynomials are equivalent, if they are conjugate with each other by a map of the form $z\mapsto z+c$. The factor space of all monic polynomials of degree $n$ factored by this equivalence relation can be identified with so called centered polynomials.
\begin{definition}
By $\mcP^n\subset\bbC[z]$ we denote the set of centered monic polynomials of degree $n$, which are monic polynomials of degree $n$ with the term of degree $n-1$ being equal to zero.
\end{definition}
Consider a monic centered polynomial of degree $n\ge 3$ and its $n-1$ distinct periodic orbits. 
If these periodic orbits are of multiplicity $1$, then by means of the Implicit Function Theorem their multipliers can be analytically continued to $n-1$ (multiple valued) algebraic functions defined on $\mcP^n$.

We present the proof of the following theorem:
\begin{theorem}\label{main_th}
For $n\ge 3$, the multipliers of any $n-1$ distinct periodic orbits considered as (multiple valued) algebraic functions on $\mcP^n$, are algebraically independent over $\bbC$. In other words, they do not satisfy any polynomial relation with complex coefficients.
\end{theorem}

A detailed proof of Theorem~\ref{main_th} (together with a proof of a slightly more general statement) is given in~\cite{mult_n}, and it relies heavily on the result of a certain computation. The goal of this paper is to give a different proof of Theorem~\ref{main_th}, in which the computational parts of the previous proof are substituted with more geometric arguments based on some ideas from one-dimensional complex dynamics.

\section{Auxiliary results and the proof of Theorem~\ref{main_th}}
We start by recalling some definitions and useful fact established in~\cite{mult_n}.

\begin{definition}
We say that a periodic orbit of a polynomial $p$ is of period $m$, if this periodic orbit consists of $m$ \emph{distinct} points. A point of period $m$ is a point that belongs to a periodic orbit of period $m$.
\end{definition}

Consider a polynomial $p\in \mcP^n$ and its $n-1$ non-multiple periodic points $z_1, \dots, z_{n-1}$ belonging to \textit{different} periodic orbits of periods $m_1,\dots, m_{n-1}$ respectively. By $\bfm$ denote the vector of periods $\bfm=(m_1,\dots,m_{n-1})$. With any such polynomial and its non-multiple periodic points belonging to different periodic orbits, one can associate the set $M^n_\bfm$ defined in the following way:
\begin{definition}\label{Mnm_def}
The set $M^n_\bfm=M^n_\bfm(p,z_1,\dots,z_{n-1})$ is the maximal irreducible analytic subset of $\mcP^n\times\bbC^{n-1}$, such that

(i) $(p,z_1,\dots,z_{n-1})\in M^n_\bfm$;

(ii) For $(q,w_1,\dots,w_{n-1})\in M^n_\bfm$, the points $w_1,\dots,w_{n-1}$ satisfy equations $q^{\circ m_j}(w_j)=w_j$, for any $j=1,2,\dots,n-1$.
\end{definition}

The following Lemma plays an important role in the proof of Theorem~\ref{main_th}. The proof of this Lemma can be found in~\cite{mult_n}.
\begin{lemma}\label{Mnkm_lemma}
Assume that $n\ge 2$.  
Then the set $M^n_\bfm$ is completely determined by the integer $n$, and the vector $\bfm$. The set $M^n_\bfm$ can be described as the closure in $\mcP^n\times\bbC^{n-1}$ of the set of all points $(p,z_1,\dots,z_{n-1})\in\mcP^n\times\bbC^{n-1}$, where $p\in\mcP^n$ and all $z_j$ are non-multiple periodic points of $p$, belonging to different periodic orbits of corresponding periods $m_j$.
\end{lemma}
\begin{remark}
In particular, since for every vector $\bfm=(m_1,\dots,m_{n-1})$ a generic polynomial has $n-1$ distinct periodic orbits of corresponding periods $m_1,\dots,m_{n-1}$, it follows from Lemma~\ref{Mnkm_lemma} that the set $M^n_\bfm$ is always non-empty.
\end{remark}

We define the multiplier map $\Lambda\colon M^n_\bfm\to\bbC^{n-1}$ that with every point $(p,z_1,\dots,z_{n-1})\in M^n_\bfm$ associates the vector of multipliers of periodic points $z_1,\dots, z_{n-1}$:
$$
\Lambda\colon (p,z_1,\dots,z_{n-1})\mapsto ((p^{\circ m_1}(z_1))', (p^{\circ m_2}(z_2))',\dots,(p^{\circ m_{n-1}}(z_{n-1}))').
$$

\begin{lemma}\label{exist_l}
For every $n\ge 2$ and every vector $\bfm\in\bbN^{n-1}$, the multiplier map $\Lambda$ is a local diffeomorphism at least at one point of $M^n_\bfm$.
\end{lemma}
Lemma~\ref{exist_l} immediately follows from~\cite[Lemma~2.1]{mult_n}, but the goal of this paper is to prove Lemma~\ref{exist_l} using different techniques.

Assuming that Lemma~\ref{exist_l} is already proved, one can deduce the proof of Theorem~\ref{main_th} by the following simple observation:

\begin{proof}[Proof of Theorem~\ref{main_th}]
It follows from Definition~\ref{Mnm_def} that $M^n_\bfm$ is an irreducible algebraic set. The multiplier map $\Lambda$ is an algebraic map on $M^n_\bfm$, and since according to Lemma~\ref{exist_l}, $\Lambda$ is a local diffeomorphism at some point of $M^n_\bfm$, it is a local diffeomorphism everywhere outside of some codimension~$1$ subset of $M^n_\bfm$.

The multipliers of $n-1$ periodic orbits of periods $m_1,\dots,m_{n-1}$ considered as (multiple valued) functions on $\mcP^n$, can be obtained from the multiplier map $\Lambda\colon M^n_\bfm\to\bbC^{n-1}$ by precomposition with a suitable inverse branch $\pi^{-1}$ of the projection $\pi\colon M^n_\bfm\to\mcP^n$. Since $\pi^{-1}$ is a local diffeomorphism everywhere outside of some codimension~$1$ subset of $\mcP^n$, the composition $\Lambda\circ\pi^{-1}$ has a non-degenerate differential at least at one point of $\mcP^n$, which implies that the multipliers of considered periodic orbits are algebraically independent (multiple valued) functions on $\mcP^n$.
\end{proof}

\section{Proof of Lemma~\ref{exist_l} (Existence of a regular point for~$\Lambda$)}
In the case when $m_1=\dots =m_{n-1}=1$, Lemma~\ref{exist_l} follows from the proposition:

\begin{proposition}\label{fixed_th} 
If $n\ge 2$ and $\bfm=(1,\dots,1)$, then the map $\Lambda$ is a local diffeomorphism at every regular point of the projection $\pi\colon M^n_\bfm\to\mcP^n$.
\end{proposition}
\begin{proof}
Assume, $(p,z_1,\dots,z_{n-1})\in M_\bfm^n$ is a regular point of the projection $\pi$. Let $\pi^{-1}$ be the inverse branch of $\pi$ in the neighborhood of $p$, such that $\pi^{-1}(p)=(p,z_1,\dots,z_{n-1})$, and define $\phi=\Lambda\circ\pi^{-1}$. Assume that the differential $d\phi_p$ vanishes at some vector $v$ from the tangent space $T_p\mcP^n$.

Let $z_n$ be the fixed point of the polynomial $p$ that is different from $z_1,\dots,z_{n-1}$. For all $q\in\mcP^n$ in a sufficiently small neighborhood of $p$, let $z_n(q)$ be the fixed point of $q$, obtained by analytic continuation of $z_n$ in that neighborhood. Define the map $\psi\colon\mcP^n\to\bbC^{n-1}\times\bbC$ by the formula
$$
\psi(q)=(\phi(q),q'(z_n(q))).
$$
According to~\cite[Section~12]{Milnor}, the multiplier $q'(z_n(q))$ can be expressed as a function of the other $n-1$ multipliers at the fixed points. Therefore, vector $v$ also lies in the kernel of the differential $d\psi_p$, but according to Theorem~1.1 from~\cite{Zarhin_2}, the kernel of $d\psi_p$ is trivial. Hence $v=0$, and $d\phi_p$ is non-degenerate.
\end{proof}

Let $\bbD$ be the open unit disk in $\bbC$. Another useful theorem formulated below, is a corollary of Theorem~5.1 from~\cite{Milnor_hyper} and concerns the case, when the preimage $\Lambda^{-1}(\bbD^{n-1})\subset M^n_\bfm$ is non-empty.
\begin{theorem}\label{attract_th}
Assume $n\ge 2$ and $\bfm\in\bbN^{n-1}$ is any vector of $n-1$ positive integers. 
If $\Lambda^{-1}(\bbD^{n-1})\subset M^n_\bfm$ is non-empty, then it consists of finitely many connected components, and $\Lambda$ is a bijection between each of these connected components and $\bbD^{n-1}$.
\end{theorem}
\begin{proof}
Assume that $\Lambda^{-1}(\bbD^{n-1})$ is non-empty and $p\in \mcP^n$ is a polynomial with $n-1$ attracting periodic orbits of periods $m_1,\dots, m_{n-1}$. According to the Fatou-Shishikura Inequality, for each of these attracting periodic orbits there exists exactly one critical point of $p$ which converges to that orbit. Since there are no more than $n-1$ critical points, this implies that the polynomial $p$ is hyperbolic and belongs to a bounded hyperbolic component $H\subset\mcP^n$.

For each $a\in\bbD$, consider a map $\mu_a\colon\bbD\to\bbD$ given by the formula
$$
\mu_a(z)=\frac{1-\overline a}{1-a}\cdot\frac{z^2-az}{1-\overline a z}.
$$
Let $\mcB$ denote the space of all such maps $\mu_a$.
Then according to Theorem~5.1 from~\cite{Milnor_hyper}, there exists a diffeomorphism $\phi\colon H\to\mcB^{n-1}$ which sends a polynomial $q\in H$ to an element $\phi(q)=(\mu_{a_1},\dots,\mu_{a_{n-1}})\in\mcB^{n-1}$ with the following property: let $\mcU_1,\dots,\mcU_{n-1}\subset\bbC$ be the periodic Fatou components of $q$ of corresponding periods $m_1,\dots,m_{n-1}$, such that each $\mcU_j$ contains a critical point of $q$. Then for each $j=1,\dots,n-1$, the restriction $\left. q^{\circ m_j}\right|_{\mcU_j}\colon\mcU_j\to\mcU_j$ is conformally conjugate to $\mu_{a_j}$.

It is easy to see that for each $a\in\bbD$, zero is a fixed point of $\mu_a$. Now we check that the map $a\mapsto\mu_a'(0)$ is a bijection of $\bbD$ to itself. Indeed,
\begin{equation}\label{mua_eq1}
\mu_a'(0)=-a\cdot\frac{1-\overline a}{1-a}=\frac{|a|^2-a}{1-a},
\end{equation}
and since $\left|\frac{1-\overline a}{1-a}\right|=1$, we have $|\mu_a'(0)|=|a|$, and~(\ref{mua_eq1}) is equivalent to
$$
a=\frac{|\mu_a'(0)|^2-\mu_a'(0)}{1-\mu_a'(0)}.
$$
The above observations imply that, 
if the map $\psi\colon\mcB^{n-1}\to\bbD^{n-1}$ is defined by the relation $\psi\colon (\mu_{a_1},\dots,\mu_{a_{n-1}})\mapsto(\mu_{a_1}'(0),\dots,\mu_{a_{n-1}}'(0))$, then the composition $\psi\circ\phi$ is a bijection between $H$ and $\bbD^{n-1}$. Since multipliers of periodic orbits are conformal invariants, the numbers $\mu_{a_1}'(0),\dots,\mu_{a_{n-1}}'(0)$ are multipliers of the corresponding attracting periodic orbits of $q$. Thus, for a suitably chosen inverse branch $\pi^{-1}$ of the projection $\pi\colon M^n_\bfm\to\mcP^n$, we have the following identity on $H$:
$$
\psi\circ\phi=\Lambda\circ\pi^{-1}.
$$
In particular, this implies that $\Lambda=\psi\circ\phi\circ\pi$ is a bijection between $\pi^{-1}(H)$ and $\bbD^{n-1}$. Since $\pi$ is a finite-to-one map, the preimage $\Lambda^{-1}(\bbD^{n-1})$ can have only finitely many connected components.
\end{proof}

The following lemma shows that the preimage $\Lambda^{-1}(\bbD^{n-1})\subset M^n_\bfm$ considered in Theorem~\ref{attract_th}, is always non-empty.
\begin{lemma}\label{exist_attr_l}
Given $n\ge 2$ and a vector $\bfm=(m_1,\dots,m_{n-1})\in\bbN^{n-1}$, the space $\mcP^n$ contains a polynomial with $n-1$ distinct attracting periodic orbits of corresponding periods $m_1,\dots,m_{n-1}$.
\end{lemma}

Now we can give a proof of Lemma~\ref{exist_l} modulo the proof of Lemma~\ref{exist_attr_l}.
\begin{proof}[Proof of Lemma~\ref{exist_l}]
Given $n\ge 2$ and a vector $\bfm\in\bbN^{n-1}$, it follows from Lemma~\ref{exist_attr_l} and Lemma~\ref{Mnkm_lemma} that there exists a point $\alpha\in M^n_\bfm$, such that $\Lambda(\alpha)\in\bbD^{n-1}$. Thus, the set $\Lambda^{-1}(\bbD^{n-1})$ is non-empty, and Theorem~\ref{attract_th} implies that $\Lambda$ bijectively maps a domain in $M^n_\bfm$ onto $\bbD^{n-1}$. Now Lemma~\ref{exist_l} follows from the fact that $\Lambda$ is an algebraic map and $\dim (M^n_\bfm)=\dim (\bbD^{n-1})=n-1$.
\end{proof}

\begin{proof}[Proof of Lemma~\ref{exist_attr_l}]
First consider the case when $m_1=\dots=m_{n-1}=1$. According to Proposition~\ref{fixed_th}, $\Lambda$ is a local diffeomorphism at least at one point of $M^n_\bfm$, and since $\Lambda$ is an algebraic map, it takes all values from $\bbC^{n-1}$ except possibly a subset of complex codimension~$1$. In particular, this means that there exists a polynomial $p\in\mcP^n$ with $n-1$ distinct attracting fixed points.

Now we will show that if there exists a polynomial $p\in\mcP^n$ with $n-1$ attracting periodic orbits of periods $m_1,\dots,m_{n-1}$ and $m_k=1$ for some $k\in\{1,\dots,n-1\}$, then there exists another polynomial $p'\in\mcP^n$ whose $n-1$ attracting periodic orbits have periods $m_1',\dots,m_{n-1}'$, where $m_j'=m_j$, for all $j\neq k$, and $m_k'$ is an arbitrary positive integer. Together with the case $m_1=\dots=m_{n-1}=1$ considered in the previous paragraph, this will prove Lemma~\ref{exist_attr_l}.

According to the Fatou-Shishikura Inequality, each critical point of the polynomial $p$ converges to a finite attracting periodic orbit, hence $p$ belongs to the connectedness locus and in particular, to a bounded hyperbolic component $H\subset\mcP^n$. 
Now according to Theorem~\ref{attract_th}, we can consider a sequence of polynomials $\{p_r\}\subset H$, $r=1, 2,\dots$, such that for all $p_r$, the $n-2$ attracting multipliers that correspond to the periodic orbits of periods $m_j$, $j\in\{1,\dots,n-1\}\setminus\{k\}$, are fixed, while the multiplier of the attracting fixed point -- the periodic orbit of period $m_k=1$, approaches the value $\exp(2\pi i/m_k')$. Since $H$ is a bounded set, it is precompact and there exists a subsequence of the sequence $\{p_r\}$ which converges to some polynomial $p_0\in\overline H$. By continuity, $p_0$ has $n-2$ attracting periodic orbits of corresponding periods $m_j$, $j\in\{1,\dots,n-1\}\setminus\{k\}$ and a fixed point $z_0$ with multiplier $\exp(2\pi i/m_k')$. A generic perturbation $q$ of $p_0$ generates a periodic orbit of period $m_k'$ from the fixed point $z_0$. Let $\lambda(q)$ be the multiplier of that periodic orbit. By analytic continuation both the periodic orbit and its multiplier can be extended to algebraic (possibly multiple valued) functions in a neighborhood of $p_0$.

It follows from the Fatou-Shishikura Inequality that $|\lambda(q)|>1$, for $q\in H$, so since $\lambda(p_0)=1$, $\lambda(q)$ is not a constant function. Therefore, according to the Maximum Modulus Principle, $|\lambda(p')|<1$ for some $p'\in\mcP^n$ arbitrarily close to $p_0$. We can choose $p'$ sufficiently close to $p_0$, so that all other attracting periodic orbits would remain attracting. This finishes the proof of Lemma~\ref{exist_attr_l}.
\end{proof}

\bibliography{mult}
\bibliographystyle{abbrv} 

\textsc{Department of Mathematics, University of Toronto, Room 6290, 40~St. George Street, Toronto, Ontario, Canada M5S~2E4}

\emph{E-mail address}: igors.gorbovickis@utoronto.ca

\end{document}